\documentclass[11pt,english]{article}
\usepackage[T1]{fontenc}
\usepackage[latin9]{inputenc}
\usepackage{textcomp}
\usepackage{amsmath}
\usepackage{amsthm}
\usepackage{amssymb}

\makeatletter
\theoremstyle{plain}
\newtheorem{thm}{\protect\theoremname}
  \theoremstyle{plain}
  \newtheorem{prop}[thm]{\protect\propositionname}
\newcommand{\lyxaddress}[1]{
\par {\raggedright #1
\vspace{1.4em}
\noindent\par}
}

\usepackage{mathrsfs}\usepackage{amsthm}\usepackage{amscd}

\usepackage{hyperref}
\usepackage{graphics}
\usepackage{a4wide}
\@ifundefined{definecolor}
 {\usepackage{color}}{}

\newcounter{EQNR}
\renewcommand{\contentsname}{Contents\\
{\footnotesize\normalfont(A table of contents should normally not
be included)}}

\makeatother

\usepackage{babel}
  \providecommand{\propositionname}{Proposition}
\providecommand{\theoremname}{Theorem}

\begin{document}

\title{Volumes of spheres and special values of zeta functions of $\mathbb{Z}$
and $\mathbb{Z}/n\mathbb{Z}$}

\author{Anders Karlsson\footnote{The first author was supported in part by the Swedish Research Council grant 104651320 and the Swiss NSF grant 200020-200400.}
\,\,\,and Massimiliano Pallich}

\date{September 6, 2022}
\maketitle
\begin{abstract}
The volume of the unit sphere in every dimension is given a new interpretation
as a product of special values of the zeta function of $\mathbb{Z}$,
akin to volume formulas of Minkowski and Siegel in the theory of arithmetic
groups. A product formula is found for this zeta function that specializes
to Catalan numbers. Moreover, certain closed-form expressions for
various other zeta values are deduced, in particular leading to an
alternative perspective on Euler's values for the Riemann zeta function.
\end{abstract}

\section{Introduction}

The determination of circumference, area, and volume of spheres is
one of the oldest topics in geometry. A more sophisticated volume
formula, due to Minkowski, is the following:
\[
\mathrm{vol(}\mathrm{SL}_{n}(\mathbb{R})/\mathrm{SL}_{n}(\mathbb{Z}))=\zeta(2)\zeta(3)...\zeta(n),
\]
where $\zeta(s)$ is the Riemann zeta function, and with a suitable
coherent choice of normalization of the Haar measure. This is not
an incidental fact, instead it is part of a more general phenomenon
discovered and developed by Siegel, Weil, Langlands, Harder, and others.
It takes the form 
\[
\mathrm{vol}(G/\Gamma)=c^{[K:\mathbb{Q}]}\prod_{i=1}^{l}\zeta_{K}(-m_{i}),
\]
and without going into details about this formula and when it holds
(see \cite{H71,P01} for more information), let us just highlight
one further example: 
\[
\mathrm{vol(}\mathrm{Sp}_{n}(\mathbb{R})/\mathrm{Sp}_{n}(\mathbb{Z}))=\zeta(2)\zeta(4)...\zeta(2n).
\]

The passage from negative integers to positive ones is done by means
of the fundamental functional equation. In the present paper we will
in particular provide a new zeta value interpretation of the $(n-1)$-dimensional
volume of spheres, which are the homogeneous spaces $S^{n-1}\cong\mathrm{SO}(n)/\mathrm{SO}(n-1)$.

As is well known, the Riemann zeta function is essentially the spectral
zeta function of the circle $\mathbb{R}/\mathbb{Z}$, more precisely
\[
\zeta(s)=\frac{1}{2}(2\pi)^{s}\zeta_{\mathbb{R}/\mathbb{Z}}(s/2)=\frac{(2\pi)^{s}}{2}\frac{1}{\Gamma(s)}\int_{0}^{\infty}\frac{1}{\sqrt{4\pi t}}\sum_{k\neq0}e^{-k^{2}/4t}t^{s/2}\frac{dt}{t},
\]
for $\mathrm{Re}(s)<1$ and then extended by meromorphic continuation.
From Fourier analysis we know that the circle and the integers are
dual groups. The function $e^{-2t}I_{0}(2t),$ with a Bessel function
appearing, is the $\mathbb{Z}$-analog of the theta series for the
circle, found inside the Mellin transform expression above. Therefore,
entirely analogously to $\mathbb{R}/\mathbb{Z}$ as is explained in
\cite{FK17}, one can define the spectral zeta function of $\mathbb{Z}$
as 
\[
\zeta_{\mathbb{Z}}(s)=\frac{1}{\Gamma(s)}\int_{0}^{\infty}e^{-2t}I_{0}(2t)t^{s}\frac{dt}{t},
\]
for $0<\mathrm{Re(s)<1/2}$ and then extended by meromorphic continuation
as will soon become clear. We then define in parallel to the above
the function
\[
Z(s)=\frac{1}{2}2\pi2^{s}\zeta_{\mathbb{Z}}(s/2),
\]
which shares with $\zeta(s)$ the property of having a functional
equation of the type of a $s\longleftrightarrow1-s$ symmetry (\cite{FK17}).
We observe the following formula very reminiscent of the above volume
expressions:
\begin{thm}
\label{prop:volume}For $n>0$, the $n$-dimensional volume of the
unit sphere in $\mathbb{R}^{n+1}$ is
\[
\mathrm{vol}(S^{n})=2\cdot Z(0)\cdot Z(-1)...\cdot Z(-n+1).
\]
\end{thm}

It turns out, by \cite{FK17}, and then by Dubout's formula in \cite{D19},
that $\zeta_{\mathbb{Z}}(s)$ is an analytic continuation essentially
of the Catalan numbers
\[
\zeta_{\mathbb{Z}}(s)=\frac{1}{4^{s}\sqrt{\pi}}\frac{\Gamma(1/2-s)}{\Gamma(1-s)}=\left(\begin{array}{c}
-2s\\
-s
\end{array}\right).
\]
The appearance of the gamma function begins explaining the volume
formula. Indeed, with the values $Z(0)=\pi$, $Z(-1)=2$ and $Z(-2)=\pi/2$
one sees that the volume formula is true for the first cases. One
could alternatively say that the Catalan numbers 
\[
C_{m}=\frac{1}{m+1}\left(\begin{array}{c}
2m\\
m
\end{array}\right),
\]
but now at the nonstandard indices $m=k/2$, appear in the classical
expression of the surface area of spheres.

The spectral zeta function of $\mathbb{Z}$ has the following beautiful
product formula:
\begin{thm}
It holds that 
\[
\zeta_{\mathbb{Z}}(s)=\prod_{k=1}^{\infty}\frac{(k-s)^{2}}{k(k-2s)}
\]
interpreted suitably when $s$ is a positive integer or half-integer.
\end{thm}

From this expression, the simple zeros and poles of $\zeta_{\mathbb{Z}}(s)$
at the positive integers and half-integers respectively, are clearly
visible. As will be shown below, it specializes to the standard product
expression for Catalan numbers:
\[
C_{m}=\prod_{k=2}^{m}\frac{m+k}{k}.
\]

In the following we shall also provide an exposition of special values
of $\zeta_{\mathbb{Z}}(s)$ and $\zeta(s)$, including correcting
certain minor inaccuracies in \cite{FK17}, as well as studying the
related and analogously defined spectral zeta function of $\mathbb{Z}/n\mathbb{Z}$,
\[
\zeta_{\mathbb{Z}/n\mathbb{Z}}(s)=\frac{1}{4^{s}}\sum_{k=1}^{n-1}\frac{1}{\sin^{2s}(\pi k/n)}.
\]
In particular, we will provide an elementary calculation of $\zeta_{\mathbb{Z}/n\mathbb{Z}}(1/2-m)$,
and indicate how then to pass to Euler's $\zeta(1-2m)$ values from
the asymptotics expansions in \cite{Si04,FK17,MV22}. After that via
the symmetry $s$ vs $1-s$ one gets $\zeta(2m)$, which leads to
$\zeta_{\mathbb{Z}/n\mathbb{Z}}(m).$ 

In this way we relate $\zeta_{\mathbb{Z}/n\mathbb{Z}}(2m/2)$ to $\zeta_{\mathbb{Z}/n\mathbb{Z}}((1-2m)/2)$.
Note that asymptotical symmetries of the type $s$ vs $1-s$ for $\zeta_{\mathbb{Z}/n\mathbb{Z}}(s/2)$
and related finite sums is not a trivial matter, in fact a certain
version of it is equivalent to various Riemann hypotheses, for $\zeta(s)$
as shown in \cite{FK17}, for certain Dirichlet $L$-functions as
proven in \cite{F16}, and for the Dedekind zeta function of the Gaussian
rationals as established in \cite{MV22}.

\textbf{Acknowledgement: }It is a pleasure to thank Fabien Friedli
for several useful remarks.

\section{Proof of the volume and product formulas}

\textbf{Proof of the volume formula. }The well-known volume (hypersurface
area) of spheres is 
\[
\mathrm{vol}(S^{n})=\frac{2\pi^{(n+1)/2}}{\Gamma((n+1)/2)},
\]
generating the numbers $2,2\pi,4\pi,2\pi^{2}$ etc. 

This leads to the following calculation
\[
\frac{\mathrm{vol}(S^{n})}{\mathrm{vol}(S^{n-1})}=\frac{2\pi^{(n+1)/2}}{\Gamma((n+1)/2)}\frac{\Gamma(n/2)}{2\pi^{n/2}}=\frac{\sqrt{\pi}\Gamma(n/2)}{\Gamma(n/2+1/2)}.
\]
This we can rewrite as follows
\[
\frac{\mathrm{vol}(S^{n})}{\mathrm{vol}(S^{n-1})}=\frac{\sqrt{\pi}\Gamma(1/2-(1/2-n/2)}{\Gamma(1-(1/2-n/2))}=\pi4^{1/2-n/2}\zeta_{\mathbb{Z}}((1-n)/2)=Z(-n+1).
\]

Since $\mathrm{vol}(S^{0})=2$ we inductively arrive at the proof
of the formula stated in Theorem \ref{prop:volume}.

\textbf{Proof of the product formula. }Recall that the Euler's beta
function has the following product formula \cite[Formula 1.1.26]{AAR99}:
\[
B(x,y)=\frac{\Gamma(x)\Gamma(y)}{\Gamma(x+y)}=\frac{x+y}{xy}\prod_{k=1}^{\infty}\frac{(1+(x+y)/k)}{(1+x/k)(1+y/k)}
\]

with poles for $x$ or $y$ equal to $0$ or a negative integer, and
analytic elsewhere. On the other hand, Dubout's expression gives
\[
\zeta_{\mathbb{Z}}(s)=\frac{\Gamma(1-2s)}{\Gamma(1-s)\Gamma(1-s)}=\frac{\Gamma(2-2s)}{(1-2s)\Gamma(1-s)\Gamma(1-s)}=\frac{(1-s)^{2}}{(1-2s)(2-2s)}\prod_{k=1}^{\infty}\frac{(1+(1-s)/k)^{2}}{(1+(2-2s)/k)}
\]

\[
=\frac{(1-s)^{2}}{(1-2s)(2-2s)}\prod_{k=1}^{\infty}\frac{(k+1-s)^{2}}{k(k+2-2s)}=\prod_{k=1}^{\infty}\frac{(k-s)^{2}}{k(k-2s)},
\]
as was to be proved. 

To see that it gives back a correct expression in the case of $s=-m$
a negative integers (the Catalan number case) note that except for
the small values, the others integer appear twice in the numerator
as well as in the denominator:
\[
C_{m}=\frac{1}{m+1}\zeta_{\mathbb{Z}}(-m)=\frac{1}{m+1}\prod_{k=1}^{\infty}\frac{(k+m)^{2}}{k(k+2m)}
\]
\[
=\frac{1}{m+1}\frac{(m+1)^{2}(m+2)^{2}...}{1\cdot2...(1+2m)(2+2m)...}=\prod_{k=2}^{m}\frac{m+k}{k},
\]
which is an expression whose validity is immediate from the definition
of $C_{m}$.

\section{Special values of $\zeta_{\mathbb{Z}}(s)$}

Here is a summary and proof of special values for $\zeta_{\mathbb{Z}}$
and its derivative. In particular, at the negative integers, $\zeta_{\mathbb{Z}}(-m)$
is rational even integral. This is not a priori obvious, but in view
of the Dubout formula for $\zeta_{\mathbb{Z}}$ this becomes clear
since they are just binomial coefficients at these points. This is
the counterpart of theorems by Hecke, Siegel, Klingen, and others
for classical zeta functions. 

The following is taken from \cite{FK17,Pa22}. First, we recall that for integers $n \geq0$, we have \begin{align*}     \Gamma(n+1) = n!     \quad \textrm{and} \quad       \Gamma \left( \frac{1}{2}+n \right)      = \frac{(2n)!}{4^n n!} \pi^{\frac{1}{2}},  \end{align*} the latter being the Legendre duplication formula. Thus, by using these formulas and the expressions above, we get for integers $n\geq 1$ that  \begin{equation} \label{spec_zeta_Z1} \begin{aligned}      & \zeta_{\mathbb{Z}}(0)=4^0 \pi^{-\frac{1}{2}} \frac{\Gamma(\frac{1}{2})}{\Gamma(1)}=1, \\     &\zeta_{\mathbb{Z}}(-n)     = 4^n \pi^{-\frac{1}{2}} \frac{\Gamma(\tfrac{1}{2}+n)}{\Gamma(1+n)}     = \frac{(2n)!}{n!n!}= \binom{2n}{n}, \\         &\zeta_{\mathbb{Z}} (-n+\tfrac{1}{2})= 4^n \pi^{-\frac{1}{2}} \frac{\Gamma(n)}{\Gamma(\tfrac{1}{2}+n)}     = \frac{4^{2n}}{2\pi n} \frac{n! n!}{(2n)!}     = \frac{4^{2n}}{2 \pi n} \binom{2n}{n}^{-1}. \end{aligned} \end{equation} 

As said before the values at positive integers and positive half-integers
are precisely the zeros and poles. 

 Now, if we differentiate $\zeta_{\mathbb{Z}}(s)$, we get \begin{align*}     \zeta_{\mathbb{Z}}'(s)     = \pi^{-\frac{1}{2}} \left( \frac{\Gamma(\frac{1}{2}-s)}{4^s \Gamma(1-s)} \right)'     &= \pi^{-\frac{1}{2}} \frac{-\Gamma(\frac{1}{2}-s)\psi_0(\frac{1}{2}-s)-\Gamma(\frac{1}{2}-s)\big(\log(4)-\psi_0(1-s)\big)}{4^s\Gamma(1-s)} \\     &= \zeta_{\mathbb{Z}}(s) \left( -\psi_0 \left( \frac{1}{2}-s \right)-2\log(2)+\psi_0(1-s) \right),     \phantom{\frac{}{}} \end{align*} since $\Gamma(\frac{1}{2}-s)'= -\Gamma(\frac{1}{2}-s) \psi_0(\frac{1}{2}-s)$ and $(4^s\Gamma(1-s))'=4^s\Gamma(1-s)(\log(4)-\psi_0(1-s))$, and where $\psi_{n}(s)$ is the polygamma function, which is defined by \begin{align*}     \psi_n(s)= \frac{d^{n}}{ds^n} \frac{\Gamma'(s)}{\Gamma(s)}. \end{align*} We can therefore deduce the following special values of $\zeta_{\mathbb{Z}}'$. 
\begin{prop}
It holds that \begin{align*}     \zeta_{\mathbb{Z}}'(0) = 0     \quad \textrm{and} \quad     \zeta_{\mathbb{Z}}' (-\tfrac{1}{2})     = \frac{8}{\pi}(1-2\log(2)), \end{align*} and \begin{align*}     & \zeta_{\mathbb{Z}}'(-n)     = \binom{2n}{n} \left( \sum_{k=1}^{n} \frac{1}{k}- \frac{2}{2k-1} \right)      && \text{for $n \geq 1$},     \\     &\zeta_{\mathbb{Z}}'(-n+\tfrac{1}{2})     = \frac{4^{2n}}{2 \pi n}\binom{2n}{n}^{-1} \left( -4\log(2) - \sum_{k=1}^{n-1} \frac{1}{k} + 2\sum_{k=1}^{n} \frac{1}{2k-1}     \right)     && \text{for $n \geq 2$}. \end{align*} 
\end{prop}

\begin{proof}
According to 6.3.2 and 6.3.3 in \cite{AS64}, we have the following values \begin{align} \label{polygamma_1}     \psi_0 \left( \frac{1}{2} \right)=-\gamma -2\log(2)      \quad \text{and} \quad      \psi_0(1)=-\gamma. \end{align} Furthermore, by formulas 6.3.2 and 6.3.4 in \cite{AS64}, we also have \begin{align}      &\psi_0 \left( n+\frac{1}{2}\right)=-\gamma -2\log(2)+2\sum_{k=1}^{n} \frac{1}{2k-1}      \qquad \text{for $n \geq 1$},  \label{polygamma_2}\\     &\psi_0(n)=-\gamma+\sum_{k=1}^{n-1}\frac{1}{k}      \qquad \qquad \qquad \qquad \qquad \quad \; \; \; \text{for $n \geq 2$}. \label{polygamma_3} \end{align} Hence, combining (\ref{polygamma_1}), (\ref{polygamma_2}), (\ref{polygamma_3}) and the special values already computed in (\ref{spec_zeta_Z1}) concludes the proof of the proposition. 
\end{proof}
Note that these values correct (confirmed by computer numerics) small
errors in \cite[Proposition 6.1]{FK17}. Indeed, the second formula
in that reference should have $-4\log2$ instead of $-4\log4$, and
the case $n=1$ needs to be interpreted correctly in view of the term
$1/(n-1)$ appearing. Finally it is stated in \cite{FK17} that $\zeta_{\mathbb{Z}}'(s)$
is zero at the positive integers, this is not true (checked by numerics)
and instead the values are given here:
\begin{prop}
Let $n$ be a positive integer, then     \begin{align*}         \zeta_{\mathbb{Z}}'(n)         = \frac{1}{n} \binom{2n}{n}^{-1}.     \end{align*} 
\end{prop}

\begin{proof}
     By the reflection formula 6.3.7 in \cite{AS64}, we have     \begin{align*}         \psi_0(1-z)=\psi_0(z)+\pi \cot( \pi z).     \end{align*}     Hence, applying this last formula, as well as (\ref{polygamma_2}) and (\ref{polygamma_3}), gives     \begin{align*}         \zeta_{\mathbb{Z}}'(n)         &= \binom{-2n}{-n} \left(- \psi_0  \left(\frac{1}{2}-s\right)-2\log(2)+ \psi_0(1-n) \right) \\         &= \binom{-2n}{-n} \left( \sum_{k=1}^{n-1}\frac{1}{k} -2 \sum_{k=1}^{n} \frac{1}{2k-1} + \pi \cot(\pi n) \right).     \end{align*}     We observe that for any positive integer $n$, we have $\binom{-2n}{-n}=0$. Therefore the two sums are eliminated by the multiplication with the binomial. However we also note that we have $\cot(\pi n)=\pm \infty$. Hence     \begin{align*}         \zeta_{\mathbb{Z}}'(n)         = \binom{-2n}{-n} \pi \cot(n\pi)         = \frac{\Gamma(1-2n)}{\Gamma(1-n)^2}\pi \cot(n\pi)         = \frac{1}{n} \binom{2n}{n}^{-1},     \end{align*}     where the last equality is obtained first by applying the reflection formula $\Gamma(1-z)\Gamma(z)=\pi\sin^{-1}(\pi z)$, then by applying the recurrence formula $\Gamma(z+1)=z\Gamma(z)$ and finally by using the double-angle formulas for sine.
\end{proof}

\section{Special values of $\zeta_{\mathbb{Z}/n\mathbb{Z}}(s)$ and $\zeta(s)$}

Sums of powers of the sine function are special cases of important
sums in works of Dedekind, Verlinde, Dowker and others in number theory
and physics, see \cite{Do92,Z96,Do15,K20}. From our perspective,
and also partly from Dowker's, they are special values of spectral
zeta functions of discrete circles. It is therefore of interest to
recall the following values, see for example \cite{Me14}. Let $n$ and $m$ be two positive integer, then we have     \begin{align*}        \zeta_{\mathbb{Z}/n\mathbb{Z}}(-m) = 2^{2m}\sum_{k=1}^{n-1} \sin^{2m} \left( \frac{\pi k}{n} \right)         =  n \sum_{k=-\lfloor \frac{m}{n} \rfloor}^{\lfloor \frac{m}{n} \rfloor} (-1)^{kn} \binom{2m}{m+kn}.     \end{align*} In the special case $m<n$, we have     
\begin{align*}       \zeta_{\mathbb{Z}/n\mathbb{Z}}(-m)  =  2^{2m}\sum_{k=1}^{n-1} \sin^{2m} \left( \frac{\pi k}{n} \right)         = n \binom{2m}{m}.     \end{align*}

Comparing the last formula to the one of $\zeta_{\mathbb{Z}}(-m)$
above (\ref{spec_zeta_Z1}), one sees a manifestation of all the trivial
zeros of the Riemann zeta function at the negative even integers,
in view of the asymptotics in \cite{Si04,FK17,MV22}
\[
\sum_{k=1}^{n-1}\frac{1}{\sin^{s}(k\pi/n)}=\frac{1}{\sqrt{\pi}}\frac{\Gamma(1/2-s/2)}{\Gamma(1-s/2)}n+2\pi^{-s}\zeta(s)n^{s}+\frac{s}{3}\pi^{2-s}\zeta(s-2)n^{s-2}+...
\]

as $n\rightarrow\infty.$ This asymptotic relation demonstrates the
intimate connection between the three zeta functions appearing in
this paper, $\zeta(s)$, $\zeta_{\mathbb{Z}}(s)$ and $\zeta_{\mathbb{Z}/n\mathbb{Z}}(s)$.

Friedli provided us with the proof of the following formula, empirically
found in \cite{Pa22}:
\begin{prop}
\label{prop:sine}Let $n$ and $m$ be positive integer, then      \begin{align*}      \zeta_{\mathbb{Z}/n\mathbb{Z}} \left(-\frac{1}{2}-m\right)    =  2^{2m+1} \sum_{k=1}^{n-1} \sin^{2m+1} \left( \frac{\pi k}{n} \right)         = 2 \sum_{j=0}^{m} (-1)^{m-j} \binom{2m+1}{j} \cot \left(  \frac{2m+1-2j}{2n}\pi \right).     \end{align*}     
\end{prop}

\begin{proof}
First we recall
\[
\sin^{2m+1}(k\pi/n)=(2i)^{-2m-1}\left(e^{i\pi k/n}-e^{-i\pi k/n}\right)^{2m+1}
\]
\[
=(2i)^{-2m-1}\sum_{j=0}^{2m+1}(-1)^{j}\left(\begin{array}{c}
2m+1\\
j
\end{array}\right)e^{i\pi k(2j-2m-1)/n}.
\]
This gives switching the finite sums
\[
\sum_{k=1}^{n-1}\sin^{2m+1}(k\pi/n)=(2i)^{-2m-1}\sum_{j=0}^{2m+1}(-1)^{j}\left(\begin{array}{c}
2m+1\\
j
\end{array}\right)\sum_{k=1}^{n-1}e^{i\pi k(2j-2m-1)/n}.
\]
The interior sum is of geometric type and can therefore be summed
and equals 
\[
-i\cdot\cot\left(\pi(2j-2m-1)/2n\right).
\]
Observe that there is a symmetry $j$ vs $2m+1-j$
\[
\cot\left(\pi(2(2m+1-j)-2m-1)/2n\right)=-\cot\left(\pi(2j-2m-1)/2n\right).
\]
Using the same symmetry for the binomial coefficients one arrives
at \begin{align*}         \sum_{k=1}^{n-1} \sin^{2m+1} \left( \frac{\pi k}{n} \right)         = \frac{1}{2^{2m}} \sum_{j=0}^{m} (-1)^{m-j} \binom{2m+1}{j} \cot \left( \left( \frac{2m+1-2j}{2n} \right) \pi \right).     \end{align*}   
\end{proof}
Thanks to the above asymptotics and the well-known series expansion
of the cotangent function 
\[
\cot(z)=\sum_{n=0}^{\infty}\frac{(-1)^{n}2^{2n}B_{2n}}{(2n)!}z^{2n-1},
\]
 one can deduce Euler's celebrated formulas that for positive integers
$m$,
\[
\zeta(-m)=(-1)^{m}\frac{B_{m+1}}{m+1},
\]
in particular $\zeta(s)$ vanishes at the negative even integers,
and $\zeta(0)=-1/2$. Indeed, setting $s=0$ in the asymptotics one
has
\[
n-1=\frac{1}{\sqrt{\pi}}\frac{\Gamma(1/2)}{\Gamma(1)}n+2\pi^{0}\zeta(0)n^{0}+...
\]
giving $\zeta(0)=-1/2.$ Specializing to $s=-1$ produces
\[
\sum_{k=1}^{n-1}\sin(k\pi/n)=\frac{1}{\sqrt{\pi}}\frac{\Gamma(1)}{\Gamma(3/2)}n+2\pi\zeta(-1)n^{-1}+\frac{-1}{3}\pi^{3}\zeta(-3)n^{-3}+...
\]
and on the other hand from Proposition \ref{prop:sine}, 
\[
\sum_{k=1}^{n-1}\sin(k\pi/n)=\cot(\pi/2n)=\frac{2n}{\pi}B_{0}+\frac{\pi}{2n}(-2)B_{2}+\frac{\pi^{3}}{8n^{3}}\frac{16B_{4}}{24}+O(1/n^{5}).
\]
This gives that 
\[
\zeta(-1)=-\frac{B_{2}}{2},\textrm{ \ensuremath{\zeta(-3)=-\frac{B_{4}}{4},...}}
\]

In view of the functional equation of $\zeta(s)$ one can then get
the values of $\zeta(2m)$ such as $\pi^{2}/6$, $\pi^{4}/90$ etc.
which appear in the volume formulas in the introduction. This, in
turn, again via the above stated asymptotics (the asymptotics for
these special $s$ here is just a formula due to the trivial zeros
$\zeta(-2m)=0$) gives the values in closed form of
\[
\zeta_{\mathbb{Z}/n\mathbb{Z}}(m)=\frac{1}{4^{s}}\sum_{k=1}^{n-1}\frac{1}{\sin^{2m}(\pi k/n)}.
\]
For example,
\[
\zeta_{\mathbb{Z}/n\mathbb{Z}}(1)=0n+\frac{2}{4\pi^{2}}\frac{\pi^{2}}{6}n^{2}+\frac{2}{12}\pi^{2-2}\left(-\frac{1}{2}\right)+0=\frac{1}{12}(n^{2}-1)
\]
and
\[
\zeta_{\mathbb{Z}/n\mathbb{Z}}(2)=\frac{1}{720}\left(n^{4}+10n^{2}+11\right).
\]

An easier approach to these last two formulas can be found in \cite{Do92,Z96}.
The values of $\zeta_{\mathbb{Z}/n\mathbb{Z}}$ at the positive half-integral
points do not have such polynomial expressions, which is related to
the elusive nature of the zeta values $\zeta(2m+1)$, $m>0.$

Finally, let us remark that the Riemann Hypothesis can be reformulated
solely in terms of a hypothetical functional symmetry of the standard
type $s\longleftrightarrow1-s$ for $\zeta_{\mathbb{Z}}(s/2)$ and
$\zeta_{\mathbb{Z}/n\mathbb{Z}}(s/2)$. This is not an incidental
fact, as it was shown to extend in \cite{F16} and \cite{MV22} to
some cases of a generalized Riemann Hypothesis.

\lyxaddress{Section de mathématiques, Université de Genève, Case postale 64,
1211 Genève, Switzerland; Mathematics department, Uppsala University,
Box 256, 751 05 Uppsala, Sweden.}

\begin{thebibliography}{AAR99}
\bibitem[AS64]{AS64} Abramowitz, Milton; Stegun, Irene A. Handbook
of mathematical functions with formulas, graphs, and mathematical
tables. For sale by the Superintendent of Documents. National Bureau
of Standards Applied Mathematics Series, No. 55 U. S. Government Printing
Office, Washington, D.C., 1964 xiv+1046 pp.

\bibitem[AAR99]{AAR99}Andrews, George E.; Askey, Richard; Roy, Ranjan,
Special functions. Encyclopedia of Mathematics and its Applications,
71. Cambridge University Press, Cambridge, 1999. xvi+664 pp.

\bibitem[Do92]{Do92}Dowker, J. S. On Verlinde's formula for the dimensions
of vector bundles on moduli spaces. J. Phys. A 25 (1992), no. 9, 2641\textendash 2648.

\bibitem[Do15]{Do15}Dowker, J.S. On sums of powers of cosecs , https://arxiv.org/pdf/1507.01848.pdf

\bibitem[D19]{D19}Dubout, Jérémy, Zeta functions of graphs, their
symmetries and extended Catalan numbers, https://arxiv.org/abs/1909.01659

\bibitem[F16]{F16}Friedli, Fabien, A functional relation for L-functions
of graphs equivalent to the Riemann hypothesis for Dirichlet L-functions.
J. Number Theory 169 (2016), 342\textendash 352.

\bibitem[FK17]{FK17}Friedli, Fabien; Karlsson, Anders, Spectral zeta
functions of graphs and the Riemann zeta function in the critical
strip. Tohoku Math. J. (2) 69 (2017), no. 4, 585\textendash 610.

\bibitem[H71]{H71}Harder, G.: A Gauss-Bonnet formula for discrete
arithmetically defined groups. Ann. Sci. Ec. Norm. Supér. 4e série,
tome 4, n\textdegree{} 3, 409-455 (1971).

\bibitem[K20]{K20}Karlsson, Anders, Spectral zeta functions. Discrete
and continuous models in the theory of networks, 199\textendash 211,
Oper. Theory Adv. Appl., 281, Birkhäuser/Springer, Cham, 2020.

\bibitem[MV22]{MV22}Meiners, Alexander; Vertman, Boris, Spectral
zeta function on discrete tori and Epstein-Riemann conjecture, https://arxiv.org/abs/2202.02420

\bibitem[Me14]{Me14}Merca, Mircea, On some power sums of sine or
cosine. Amer. Math. Monthly 121 (2014), no. 3, 244\textendash 248.

\bibitem[Pa22]{Pa22}Pallich, Massimiliano, The Riemann Zeta Function
and the Spectral Zeta Function of Graphs, Master Thesis, University
of Geneva, 2022

\bibitem[P01]{P01}Parshin, A. N. Note on the Siegel formula. Algebraic
geometry, 11. J. Math. Sci. (New York) 106 (2001), no. 5, 3336\textendash 3339.

\bibitem[Si04]{Si04}Sidi, Avram, Euler-Maclaurin expansions for integrals
with endpoint singularities: a new perspective. Numer. Math. 98 (2004),
no. 2, 371\textendash 387.

\bibitem[Z96]{Z96}Zagier, Don Elementary aspects of the Verlinde
formula and of the Harder- Narasimhan-Atiyah-Bott formula. Proceedings
of the Hirzebruch 65 Conference on Algebraic Geometry (Ramat Gan,
1993), 445\textendash 462, Israel Math. Conf. Proc., 9, Bar-Ilan Univ.,
Ramat Gan, 1996.
\end{thebibliography}
\end{document}